\documentclass[a4paper,11pt,reqno]{amsart}
\usepackage{color}
\usepackage{enumerate}
\usepackage{amssymb, amsmath}

\theoremstyle{plain}
\newtheorem{theorem}{Theorem}[section]
\newtheorem{corollary}[theorem]{Corollary}
\newtheorem{lemma}[theorem]{Lemma}
\newtheorem{proposition}[theorem]{Proposition}
\newtheorem{definition}[theorem]{Definition}

\newtheorem*{definition*}{Definition}

\theoremstyle{remark}
\newtheorem{remark}[theorem]{Remark}
\newtheorem{example}[theorem]{Example}

\newtheorem*{claim*}{Claim}
\newtheorem*{remark*}{Remark}
\newtheorem*{example*}{Example}
\newtheorem*{notation*}{Notation}

\begin{document}
\title[Metric Measure Spaces with Variable Ricci Bounds]{Metric Measure Spaces with Variable Ricci Bounds and Couplings of Brownian Motions}
\author{Karl-Theodor Sturm}
 \address{
 University of Bonn\\
 Institute for Applied Mathematics\\
 Endenicher Allee 60,
 53115 Bonn\\
 Germany}
  \email{sturm@uni-bonn.de}

\maketitle

\def\N{{\mathbb N}}
\def\Z{{\mathbb Z}}
\def\R{{\mathbb R}}
\def\C{{\mathbb C}}
\def\D{{\mathsf D}}
\def\Q{{\mathbb Q}}
\def\QQ{{\mathbf Q}}
\def\PP{{\mathbf P}}
\def\EE{{\mathbb E}}
\def\DD{{\mathbb D}}

\def\Pr{{\mathcal P}}
\newcommand{\K}{\mathrm{k}}
\renewcommand{\k}{\mathrm{k}}
\newcommand{\KK}{\mathrm{K}}
\newcommand{\LL}{\mathrm{L}}
\newcommand{\RR}{\mathrm{R}}
\newcommand{\HH}{\mathrm{H}}
\newcommand{\Hess}{\mathrm{Hess}}

\newcommand{\Dom}{\mathit{Dom}}
\newcommand{\A}{\mathcal{A}}
\newcommand{\B}{\mathcal{B}}
\newcommand{\E}{\mathcal{E}}
\newcommand{\BE}{\mathrm{BE}}
\newcommand{\CD}{\mathrm{CD}}
\newcommand{\EVI}{\mathrm{EVI}}
\newcommand{\Ric}{\mathrm{Ric}}
\newcommand{\diam}{\mathrm{diam}}
\newcommand{\tr}{\mathrm{tr}}
\newcommand{\Ent}{\mathrm{Ent}}
\numberwithin{equation}{section}

\begin{abstract}
The goal of this paper is twofold: we study metric measure spaces $(X,d,m)$ with variable lower bounds for the Ricci curvature and
we study pathwise coupling of Brownian motions.
Given any lower semicontinuous function $\k:X\to\R$ we introduce the curvature-dimension condition $\CD(\k,\infty)$ which canonically extends the curvature-dimension condition $\CD(\KK,\infty)$ of Lott-Sturm-Villani for constant $\KK\in\R$. For infinitesimally Hilbertian spaces we prove \begin{itemize}
\item
its equivalence to an evolution-variation inequality $\EVI_\k$ which in \\ turn extends the $\EVI_\KK$-inequality of Ambrosio-Gigli-Savar\'e;
\item
its stability under convergence and its local-to-global property.
\end{itemize}
For metric measure spaces with uniform lower curvature bounds $\KK$ we prove that for each pair of initial distributions $\mu_1,\mu_2$ on $X$ there exists a coupling $B_t=(B_t^1,B_t^2)$, $t\ge0$, of two Brownian motions on $X$ with the given initial distributions such that a.s.
\begin{equation*}d\big(B^1_{s+t},B^2_{s+t}\big)\le e^{-\KK t/2}\cdot d\big(B_s^1,B_s^2\big)\qquad (\forall s,t\ge0).\end{equation*}
\end{abstract}

\section{Heat Flow on Metric Measure Spaces}

Throughout this paper,  a \emph{metric measure space} will be triple $(X,d,m)$ where
$(X,d)$ is a complete, separable metric space and $m$ is a measure on $X$ equipped with its  Borel $\sigma$-field $\B(X)$.
To simplify the presentation, we also assume in addition that $d$ is a length metric, that $m$ has  full topological support
and that the following weak integrability property holds
\begin{equation}\int_X e^{-C\cdot d^2(x',x)}\,dm(x)<\infty\label{gaussian}\end{equation}
(for some  $x'\in X$ and $C\in\R$).
$\Pr(X)$ will denote the space of Borel probability measures on $X$.
There are two canonical ways to define the \emph{heat flow} on a mms $(X,d,m)$
\begin{itemize}
\item either as the gradient flow for the \emph{energy} $\E$ in the Hilbert space $L^2(X,m)$
\item or as the gradient flow for the \emph{entropy} $\Ent$ in the Wasserstein  space $\Pr_2(X)$.
\end{itemize}
In the sequel, we will briefly present both approaches and we will illustrate
that in great generality both approaches will coincide.

\subsection{Eulerian Approach via Energy}

Given any function $f: X\to\R$ we define its pre-energy by
\begin{equation*}\E^0(f):=
\left\{\begin{array}{cc}
\int_X |Df|^2\,dm,&\mbox{ if $f$ is bounded and Lipschitz}\\
+\infty,&\mbox{ else}
\end{array}
\right.
\end{equation*}
and its \emph{energy} -- often called Cheeger energy -- as the relaxation or lower semicontinuous envelop of the pre-energy:
\begin{equation*}\E(f)=\liminf_{g\to f \mbox{ in } L^2}\E^0(g).\end{equation*}
Here $|Df|$ denotes the  metric slope (or local Lipschitz constant) of a Lip\-schitz function defined as
\begin{equation*}|Df(x)|:=\limsup_{y\to x}\frac{|f(y)-f(x)|}{d(y,x)}.\end{equation*}

The energy $\E$ is a lower semicontinuous convex functional on $L^2(X,m)$. It is  2-homogeneous but not necessarily quadratic.
Its domain $\Dom(\E):=\{f\in L^2(X,m):\, \E(f)<\infty\}$ is a dense linear subspace of $L^2(X,m)$ \cite{AGS-Calc}, Prop. 4.1.
For each $f\in\Dom(\E)$ there exists the minimal weak upper gradient $|Df|_w$,  a unique element of minimal norm in $\{g\in L^2(X,m): \ \exists f_n\in \mathrm{Lip}_b(X), f_n\to f, |Df_n|\rightharpoonup g\}$. It satisfies $\E(f)=\int |Df|_w^2\,dm$.

\begin{proposition}[\cite{AGS-Calc}, Chapter 4]
\begin{itemize}
\item
The gradient flow for the energy $\E$ in $L^2(X,m)$ defines uniquely a continuous semigroup $(T_t)_{t\ge0}$ of contractions in $L^2(X,m)$.
\item
For each $f\in L^2(X,m)$ the trajectory $t\mapsto T_tf$ is continuous in $t\in[0,\infty)$ and locally Lipschitz continuous in $t\in(0,\infty)$.
\item The heat flow is mass preserving: for each $f\in L^2(X,m)\cap L^1(X,m)$
\begin{equation*}\int T_t f\,dm=\int f\,dm.\end{equation*}
\item It is contracting in $L^p$: for each $p\in[1,\infty]$ and each  $f,g\in L^2(X,m)\cap L^p(X,m)$
\begin{equation*}\|T_tf-T_tg\|_p\le \|f-g\|_p.\end{equation*}
\end{itemize}
\end{proposition}

In more details, the heat flow $f_t=T_tf_0$ is defined via the differential inclusion
\begin{equation*}\frac{d}{dt}f_t\in -\partial^- \E(f_t)\end{equation*}
where $\partial^-\E$ denotes the subdifferential of the convex function $\E$.
The Laplacian is defined for those $f\in L^2$ with
$\partial^- \E(f)\not=\emptyset$ as the the element of minimal $L^2$-norm within $-\partial^- \E(f)$.

In general, the heat semigroup -- and equivalently the Laplacian -- will neither be linear nor $m$-symmetric.
The heat flow is linear if and only if the energy is a quadratic functional. And in this case, the heat semigroup will also be $m$-symmetric.
Note that by this construction, the heat flow will be the solution to $\frac{d}{dt}f_t=\Delta f_t$ -- whereas in parts of the literature it is regarded as the solution  to $\frac{d}{dt}f_t=\frac12\Delta f_t$.

\begin{example}[\cite{OhSt1, OhSt2}]
Let $(M,F,m)$ be a smooth Finsler space. Then the associated heat flow on $M$ is linear if and only if the norm $F$ on each tangent space is Hilbertian, or in other words, if and only the manifold is Riemannian.

Even if the heat semigroup on Finsler spaces  is non-linear  it shares many properties with the linear heat semigroup for Dirichlet forms, e.g. integrated Gaussian estimates \`a la Davies, pointwise comparison \`a la Cheeger-Yau and gradient estimates \`a la Bakry-Emery-Ledoux.
\end{example}

\subsection{Lagrangian Approach via Entropy}

For $p\in[1,\infty)$ we define the
$L^p$-Wasserstein distance between $\mu_0,\mu_1\in\Pr(X)$ by
\begin{align*}
  W_p(\mu_0,\mu_1)~=~\left(\inf \int d(x,y)^p\, dq(x,y)\right)^{1/p}
\end{align*}
and
$W_\infty(\mu_0,\mu_1)~=~\inf\| d(.)\|_{L^\infty(X^2,q)}$
where in both cases the infimum is taken over all Borel probability measures $q$ on
$X\times X$ with marginals $\mu_0$ and $\mu_1$.
Note that
$W_\infty(\mu_0,\mu_1)=\lim_{p\to\infty} W_p(\mu_0,\mu_1)$.

The case $p=2$ will be of particular interest for us.
We denote by $\Pr_2(X)$ the $L^2$-Wasserstein space over
$(X,d)$, i.e. the set of all Borel probability measures $\mu$
satisfying $\int_X d(x_0,x)^2\mu(d x)<\infty$
for some, hence any, $x_0\in X$.

Given a measure $\mu\in\Pr_2(X)$ we define its \emph{relative entropy} or \emph{Boltzmann entropy} by
\begin{align*}
  \Ent(\mu)~:=~\int \rho\log\rho \, dm\;,
\end{align*}
if $\mu=\rho m$ is absolutely continuous w.r.t. $m$ and
$(\rho\log\rho)_+$ is integrable. Otherwise we set $\Ent(\mu)=+\infty$.

\begin{definition}
Given a number $\KK\in\R$  we say that $(X,d,m)$ satisfies the \emph{curvature-dimension condition} $\CD(\KK,\infty)$ if the Boltzmann entropy
is $\KK$-convex on the $L^2$-Wasserstein space $(\Pr_2(X), W_2)$.
\end{definition}
Here a function $S$ on a metric space $(Y,d_Y)$ is called $\KK$-convex if every pair of points $y_0,y_1\in Y$ can be joined by a (minimizing, constant speed) geodesic $\big(y_t\big)_{0\le t\le1}$ in $Y$ such
\begin{equation*}S(y_t)\le (1-t)\,S(y_0)+t\, S(y_1)-\frac \KK 2 t(1-t)\, d_Y(y_0,y_1)^2\end{equation*}
for all $t\in [0,1]$.
(The latter can equivalently be expressed by the fact that the function $u(t)=S(y_t)$ is upper semicontinuous in $t\in[0,1]$, continuous in $(0,1)$ and satisfies
$u''\ge \KK  |\dot y|^2$
weakly in $(0,1)$.)
In the case $\KK=0$ it is the classical convexity.
In the general case, $\KK$-convexity gives a precise meaning for weak solutions to the differential inequality
$D^2 S\ge K$
on geodesic spaces.

\begin{proposition}[\cite{Stu1}, \cite{LV}, \cite{AGS-BE}]
The  curvature-dimension condition $\CD(\KK,\infty)$ has important stability and transformation properties:
\begin{itemize}
\item it is preserved under convergence of the underlying mms (with respect to mGH convergence or $\DD$-convergence as well as with respect to the `pointed versions' of these convergence concepts);
\item it has the local-to-global property (provided the space is non-branching);
\item it has the tensorization property (provided all the spaces are non-branching).
\end{itemize}
\end{proposition}

\begin{proposition}[\cite{LV}, \cite{AGS-BE}]
The  curvature-dimension condition $\CD(\KK,\infty)$ for  $\KK>0$ implies various functional inequalities,
each of them with sharp constants,
\begin{itemize}
\item spectral gap estimate
\item Talagrand inequality
\item logarithmic Sobolev inequality.
\end{itemize}
\end{proposition}

\begin{proposition}[\cite{AGS-Calc}] Assume that the condition $\CD(\KK,\infty)$ holds true for some $\KK\in\R$.
\begin{itemize}
\item[\bf(i)]
For every $\mu\in\Pr_2(X)$ with $\Ent(\mu)<\infty$ there exists a unique gradient flow $(P_t\mu)_{t\ge0}$
for the Boltzmann entropy $\Ent$ in the $L^2$-Wasserstein space $(\Pr_2(X), W_2)$, starting in $\mu$.
\item[\bf(ii)] For each $\mu\in\Pr_2(X)$ with $\Ent(\mu)<\infty$ and with $f=\frac{d\mu}{dm}\in L^2(X,m)$ the gradient flow of the entropy and the previously defined gradient flow of the energy coincide:
    \begin{equation*}P_t(f\,m)=(T_tf)\,m.\end{equation*}
\end{itemize}
\end{proposition}

\subsection{Wasserstein Contraction, Gradient Estimates, and Bochner's Formula}

>From now on, we will in addition always assume that the metric measure spaces $(X,d,m)$ under consideration will be \emph{``infinitesimally Hilbertian''} in the sense that the canonical energy (as introduced above) is \emph{quadratic}:
\begin{equation*}\E(u+v)+\E(u-v)=2\E(u)+2\E(v)\qquad (\forall u,v\in \Dom(\E)).\end{equation*}
To simplify the presentation, we also assume
that
every bounded function $f\in\Dom(\E)$ with $|Df|_w\le1$ admits a continuous representative
(``property C'').
Note that property (ii) in the subsequent theorem already implies that the energy is quadratic. Moreover, each of properties (i) and (ii) imply the above mentioned property C.

\begin{theorem}[\cite{AGS-BE}]\label{fund} For any mms $(X,d,m)$ as above, the following properties are equivalent
\begin{itemize}
\item[\bf(i)] Curvature-dimension condition $\CD(\KK,\infty)$;

\item[\bf(ii)] Evolution-variation inequality $\EVI_\KK$: for every $\mu_0\in\Pr_2(X)$ there exists a curve $(\mu_t)_{t>0}$ in $\Dom(\Ent)$ with $\lim_{t\to0}\mu_t=\mu_0$ such that $\forall\nu\in\Pr_2(X), \forall t>0$
    \begin{equation*}\frac{d^+}{dt}\frac12 W_2^2(\mu_t,\nu)+\frac\KK2 W^2_2(\mu_t,\nu)\le\Ent(\nu)-\Ent(\mu_t);\end{equation*}

\item[\bf(iii)] $L^2$-Wasserstein contraction: $\forall f_1,f_2\in L^2(X,m), \forall t\in\R_+$:
\begin{equation*}W_2\Big(T_t(f_1\,m), T_t(f_2\,m)\Big)\le e^{-Kt}\, W_2\Big((f_1\,m),(f_2\,m)\Big);\end{equation*}

\item[\bf(iv)] $L^2$-gradient estimate: $\forall u\in\Dom(\E)$ and $\forall t>0$
\begin{equation*}
|D T_tu|_w^2\le e^{-2\KK t}T_t(|Du|_w^2);
\end{equation*}

\item[\bf(v)]  Bochner's inequality or Bakry-Emery condition $\BE(\KK,\infty)$:
 $\forall u\in\Dom(\Delta)$ with $\Delta u\in\Dom(\E)$
and $\forall \phi\in\Dom({\Delta})\cap L^\infty(X,m)$ with $\phi\ge0$, ${\Delta}\phi\in L^\infty(X,m)$
\begin{equation}\label{bochner1}
\int \left(\frac12\Delta -\KK\right)\, \phi\cdot |Du|_w^2\,dm \ge\int \phi \cdot \langle\nabla u,\nabla\Delta u\rangle\,dm.
\end{equation}

\end{itemize}
\end{theorem}

\begin{corollary}[\cite{AGS-Calc}, \cite{RaSt}]
Assume that some/all of the above properties are satisfied. Then
\begin{itemize}
\item
$(\E,\Dom(\E))$ is a strongly local, quasi-regular Dirichlet form.
\item
Its carr\`e du champ operator coincides with the squared minimal weak upper gradient: $\Gamma(u)=|Du|_w^2$.
\item
The generator of the Dirichlet form is
the negative Laplacian $-\Delta$, introduced as the single-valued subdifferential of $\E$. It is a linear self-adjoint operator.
\item
The $\EVI_\KK$-curve in (ii) coincides with the heat flow: $\mu_t=P_t\mu_0$.
\item
For every $\mu_0,\mu_1\in\Pr_2(X)$ with $\mu_0,\mu_1\ll m$ the geodesic connecting them is unique.
\end{itemize}
\end{corollary}

The previous theorem is of fundamental importance for our understanding of curvature bounds on metric measure spaces and it has  plenty of applications. It also allows for various extensions and improvements.
One direction of improvement is to combine the curvature bound $\KK$ with a dimension bound $N$ which leads to  the curvature-dimension condition $\CD(\KK,N)$ introduced in \cite{Stu2} and also studied (in slightly modified form) in \cite{Vi2}, see also \cite{BaSt}, \cite{CaSt}.
Erbar-Kuwada-Sturm \cite{EKS} succeeded to formulate proper versions (taking into account the additional information of the upper dimension bound $N<\infty$) of each of the properties in the previous theorem
 -- among them
$\EVI_{\KK,N}$ and $\BE(\KK,N)$ inequalities -- and to prove their equivalence.
The $\CD(\KK,N)$-condition for finite $N$ also has the advantage that it allows to deduce curvature bounds under time change of the process and/or conformal transformation of the metric \cite{Stu3}.

Another direction will be to consider variable  curvature bounds instead of uniform bounds. This will be the topic of the final chapter 3 of this paper. There we will introduce and study appropriate modifications of the previous properties (i), (ii), (iv) and (v) for
non-constant curvature bounds $\k:X\to\R$. For a more refined approach -- in a more `regular' setting however -- which even allows to define a \emph{Ricci tensor}, see \cite{Stu3}.

Yet another direction of improving the previous theorem consists in studying $L^p$-versions instead of $L^2$-versions.
Using a remarkable  self-improvement property of Bochner's inequality in the `smooth' $\Gamma_2$-setting (i.e. assuming the existence of a nice algebra of functions), Bakry \cite{Bak} deduced $L^p$ versions of the gradient estimate (iv).  Savar\'e \cite{Sav} extended this argument to the non-smooth setting and combined it with Kuwada's duality argument \cite{Kuw, Kuw2} to obtain an $L^p$-Wasserstein contraction estimate in the spirit of \cite{ReSt}.

\begin{theorem}[\cite{Sav}]
\label{savare}
Assume that $(X,d,m)$ satisfies CD$(K,\infty)$. Then for all $p\in[1,\infty]$,  all $\mu_1,\mu_2\in\Pr(X)$ and all $t\in\R_+$
\begin{equation*}W_p(P_t\mu_1, P_t\mu_2)\le e^{-Kt}\, W_p(\mu_1,\mu_2).\end{equation*}
\end{theorem}
This result will be the key ingredient for the construction of coupled pairs of Brownian motions on $X$ which is the content  of the subsequent chapter.

\section{Coupled Pairs of Brownian Motions on $X$}

\subsection{Brownian Motion on $X$}

As a consequence of the previous theorem (the classical $p=2$ version suffices), the gradient flow $(P_t)_{t>0}$ extends to all of $\Pr_2(X)$ such that for fixed $t>0$ the mapping $\mu\mapsto P_t\mu$ is  Lipschitz continuous (w.r.t. the $W_2$-metric).
Let us put $p_t(x,.)=P_t\delta_x$. Then the mapping $x\mapsto p_t(x,.)$ from $X$ to $\Pr_2(X)$ is Lipschitz continuous, in particular, Borel measurable. Thus $p_t(.,.)$ is a Markov kernel on $(X,\B(X))$.

Uniqueness of the gradient flow implies $P_s(P_t\mu)=P_{s+t}\mu$ for all $\mu$ with finite entropy; by continuity in $\mu$ this extends to all of $\Pr_2(X)$. Thus $(p_t)_{t\ge0}$ is a semigroup of Markov kernels on $(X,\B(X))$. Note that each of the measures $p_t(x,.)$ for $x\in X$ and $t>0$ is absolutely continuous w.r.t. $m$.

Since the Dirichlet form $\E$ on $L^2(X,m)$ is quasi-regular and local, there exists an $m$-invariant continuous strong Markov process which (or more precisely, the transition semigroup of which) is $m$-equivalent to the semigroup $(p_t)_{t\ge0}$. Since the latter is absolutely continuous w.r.t. $m$ the continuous Markov process can be chosen to be equivalent to $(p_t)_{t\ge0}$  (\cite{FOT}, Theorems 4.5.1 and 4.5.4).

This continuous stochastic process can be obtained as follows:
Given the Markov semigroup $(p_t)_{t\in\R_+}$ on $(X,\B(X))$, an arbitrary initial distribution $\mu\in\Pr(X)$ and a finite subset $J=\{t_1,\ldots,t_r\}$ of $\R_+$ we define the \emph{finite dimensional distribution} $P_J^\mu$ --
 a probability measure on $X^r$ -- as follows
\begin{equation*}P_J^\mu(B_1\times\ldots B_r)= \int_{X} \int_{B_1}\ldots \int_{B_r} p_{t_r-t_{r-1}}(x_{r-1},dx_r)\, \ldots \, p_{t_1}(x_{0},dx_1)\,
\mu(dx_0).\end{equation*}
The probability measures $\{P_J^\mu:\ J \mbox{ finite }\subset\R_+\}$ constitute a \emph{consistent family} which implies that their projective limit
\begin{equation*}\PP^\mu_{\R_+}=\lim_{\leftarrow}P_J^\mu\end{equation*}
exists: it is a probability measure on $(X^{\R_+},\B(X)^{\R_+})$ with the property that
\begin{equation*}(\pi_J)_*\PP^\mu_{\R_+}=P_J^\mu\end{equation*}
for each finite $J\subset\R_+$ where $\pi_J$ denotes the projection $\omega\mapsto(\omega(t_1),\ldots,\omega(t_r))$ from $X^{\R_+}$ onto $X^r$.
The measures $\PP^x_{\R_+}:=\PP^{\delta_x}_{\R_+}$  depend in a measurable way on $x$ and
\begin{equation*}\PP^\mu_{\R_+}(.)=\int_X \PP^x_{\R_+}(.)\,\mu(dx)\end{equation*}
for any $\mu\in\Pr(X)$ (cf. \cite{Bauer}).

There are various ways to construct  a continuous modification out of this stochastic process.
Firstly, since we already know that there exists a continuous modification we may refer to a result of Doob which states that in this case the subset
${\mathcal C}(\R_+,X)$ has full outer measure w.r.t. $\PP^\mu_{\R_+}$. This allows to define a probability measure $\PP^\mu$ on
${\mathcal C}(\R_+,X)$ equipped with the trace $\sigma$-field $\B(X)^{\R_+}\cap {\mathcal C}(\R_+,X)$ by
\begin{equation*}\PP^\mu(A\cap{\mathcal C}(\R_+,X)):=\PP^\mu_{\R_+}(A)\qquad(\forall A\in \B(X)^{\R_+}).\end{equation*}
The process $(\pi_t)_{t\ge0}$ on $({\mathcal C}(\R_+,X),\B(X)^{\R_+}\cap {\mathcal C}(\R_+,X),\PP^\mu)$ will do the job
(\cite{Bauer}, Theorem 63.2 and Lemma 63,8) and the process $(B_t)_{t\ge0}:=(\pi_{t/2})_{t\ge0}$ then will be a Brownian motion on $X$ with initial distribution $\mu$.

Alternatively, we may
restrict the given process $(\pi_t)_{t\in\R_+}$ on $(X^{\R_+},\B(X)^{\R_+},\PP^\mu_{\R_+})$
to dyadic time instances, i.e. replace $\R_+$ by the set $\D:=\left\{k\cdot 2^{-n}:\ k,n\in\N_0\right\}$.
Then extend this process (with time parameter $\D$) by local uniform continuity -- outside of a zero set -- to a process with time parameter $\R_+$.
Indeed, for $\PP^\mu_{\R_+}$-a.e. $\omega$
 the limit
\begin{equation*}B_t(\omega)=\lim_{s\to t/2, s\in\D} \pi_s(\omega)\end{equation*}
exists for each $t\in\R_+$ and   the trajectory $t\mapsto B_t(\omega)$ is continuous.
Moreover, for each $t\in\R_+$ the random variables  $\pi_{t/2}$ and $B_t$ coincide $\PP^\mu_{\R_+}$-a.s.
(Cf. \cite{Bauer}, Lemma 63.5 and subsequent Remarks 1+2).
Note that for the construction of the process $(B_t)_{t\in\R_+}$ the measures $\PP^\mu_{\R_+}$ may be replaced by its projection onto the space $X^\D$.

\begin{proposition}
Given any initial distribution $\mu$ on $X$, the stochastic process $(B_t)_{t\in\R_+}$ with values in $X$, constructed as above on the probability space $(X^\D, \B(X)^\D, \PP^\mu_\D)$,
is a Brownian motion on $X$ with initial distribution $\mu$.
\end{proposition}

Recall that by convention a Brownian motion on $X$ has generator $\frac12\Delta$. Note that the link to the semigroup $(T_t)_{t\ge0}$ -- which we defined as the semigroup of selfadjoint operators on $L^2(X,m)$ with generator $\Delta$ -- is given by
\begin{equation*}T_tu(x)=\EE_x\left[u(B_{2t})\right]\end{equation*}
for $m$-a.e. $x\in X$, each $t>0$ and each Borel measurable $L^2$-function $u$ on $X$.

\subsection{Coupled Semigroups in Discrete Time}

For the sequel, we have to introduce additional notation.
\begin{equation*}\B^u(X^2):=\bigcap_{\alpha\in\Pr(X^2)}\B^\alpha(X^2)\end{equation*} will denote the $\sigma$-field of \emph{universally measurable} subsets of $X^2$. It is the intersection of all the $\B^\alpha(X^2)$ where $\alpha$ runs through the set $\Pr(X^2)$ of all Borel probability measures on $X^2$ and where $\B^\alpha(X^2)$ is the completion of the Borel $\sigma$-field on $X^2$ w.r.t. $\alpha\in\Pr(X^2)$.
Moreover, $\D:=\left\{k\cdot 2^{-n}:\ k,n\in\N_0\right\}$
will denote the set of nonnegative \emph{dyadic} numbers whereas
$\D_n:=\left\{k\cdot 2^{-n}:\ k\in\N_0\right\}$ for fixed $n\in\N_0$.

\begin{lemma} For each $t\in\R_+$ there exists a Markov kernel $q_t^*$ on $(X^2,\B^u(X^2))$ with the following properties:
\begin{itemize}
\item[\bf(i)] For each $(x,y)\in X^2$ the probability measure $q_t^*\Large((x,y),.\Large)$ is a coupling of the probability measures $p_t(x,.)$ and $p_t(y,.)$.
\item[\bf(ii)] For each $(x,y)\in X^2$ and $q_t^*\Large((x,y),.\Large)$-a.e. $(x',y')\in X^2$
\begin{equation*}d(x',y')\le e^{-Kt}\cdot d(x,y).\end{equation*}
\end{itemize}
\end{lemma}

\begin{proof}
Applying  Theorem \ref{savare} with $p=\infty$ to $\mu_1=\delta_x, \mu_2=\delta_y$ yields the existence of  at least one probability measure $q_t^*((x,y),.)$ with properties (i) and (ii) for each $x,y\in X^2$.
The class of all these measures is closed (for given $x,y$).
Thus according to a classical measurable selection theorem (see e.g. \cite{Boga} , Thm. 6.9.2), we may choose the optimal coupling satisfying (i) and (ii) in such a way that
the map
\begin{equation*}(x,y)\mapsto q_t^*((x,y),.),\qquad \large(X^2,\B^u(X^2)\large)\to \left(\Pr(X^2), \B\left(\Pr(X^2)\right)\right)\end{equation*}
is measurable.
(Note that the $\sigma$-field of universally measurable subsets of $X^2$ contains all Souslin subsets of $X^2$.)
\end{proof}

\begin{lemma} For each $n\in\N_0$ there exists a Markov semigroup $(q_t^{(n)})_{t\in\D_n}$ on $(X^2,\B^u(X^2))$ with the following properties:
\begin{itemize}
\item[\bf(i)] For each $t\in\D_n$ and each $(x,y)\in X^2$ the probability measure $q_t^{(n)}\Large((x,y),.\Large)$ is a coupling of the probability measures $p_t(x,.)$ and $p_t(y,.)$.
\item[\bf(ii)] For each $t\in\D_n$, each $(x,y)\in X^2$ and $q_t^{(n)}\Large((x,y),.\Large)$-a.e. $(x',y')\in X^2$
\begin{equation*}d(x',y')\le e^{-Kt}\cdot d(x,y).\end{equation*}
\end{itemize}
\end{lemma}

\begin{proof}
Given $t\in\D_n$, say $t=k\cdot 2^{-n}$ we put
\begin{equation*}q_t^{(n)}:=\underbrace{q_{2^{-n}}^*\circ\ldots\circ q_{2^{-n}}^*}_{k \ \mathrm{times}}.\end{equation*}
This obviously defines for each $t\in\D_n$ a Markov kernel on $(X^2,\B^u(X^2))$ and $q_s^{(n)}\circ q_t^{(n)}=q_{s+t}^{(n)}$ for all $s,t\in\D_n$.
Moreover, properties (i) and (ii) of the theorem are inherited (by iteration) from the corresponding properties (i) and (ii) for the kernel
$q_{2^{-n}}^*$ (see previous lemma).
Indeed, for each $i=1,\ldots,k$, each  $(x_{i-1},y_{i-1})\in X^2$ and $q_{2^{-n}}^{(n)}\Large((x_{i-1},y_{i-1}),.\Large)$-a.e. $(x_i,y_i)\in X^2$
\begin{equation*}d(x_i,y_i)\le e^{-K2^{-n}}\cdot d(x_{i-1},y_{i-1}).\end{equation*}
This yields (ii) for the kernel $q_t$ with $t=k 2^{-n}$.

Moreover, for each bounded Borel function $f$ on $X$
\begin{eqnarray*}
\lefteqn{\int_{X^2} f(x_k)\, q^{(n)}_{k2^{-n}}((x_0,y_0),d(x_k,y_k))}\\
&=&
\int_{X^2}\left[\int_{X^2}f(x_k)\,q^*_{2^{-n}}((x_{k-1},y_{k-1}),d(x_k,y_k))\right] \\
&&\qquad\qquad\qquad\qquad\qquad\qquad(q^*_{2^{-n}})^{k-1}((x_0,y_0),d(x_{k-1},y_{k-1})) \\
&=&
\int_{X^2}\left[\int_{X^2}f(x_k)\,p_{2^{-n}}(x_{k-1},dx_k)\right] (q^*_{2^{-n}})^{k-1}((x_0,y_0),d(x_{k-1},y_{k-1})) \\
&=&\ldots\\
&=&
\int_{X^2}\left[\int_{X^2}f(x_k)\,p_{2\cdot 2^{-n}}(x_{k-2},dx_k)\right] (q^*_{2^{-n}})^{k-2}((x_0,y_0),d(x_{k-2},y_{k-2})) \\
&=&\ldots\\
&=&
\int_{X^2}\left[\int_{X^2}f(x_k)\,p_{(k-1)\cdot 2^{-n}}(x_{1},dx_k)\right] q^*_{2^{-n}}((x_0,y_0),d(x_{1},y_{1})) \\
&=&\int_X f(x_k)\, p_{k\cdot 2^{-n}}(x_{0},dx_k).
\end{eqnarray*}
Similarly,
\begin{equation*}\int_{X^2} f(y_k)\, q^{(n)}_{k2^{-n}}((x_0,y_0),d(x_k,y_k))=\int_X f(y_k)\, p_{k\cdot 2^{-n}}(y_{0},dy_k).\end{equation*}
This proves property (i).
\end{proof}

\begin{remark}
For given $(x,y)\in X^2$ and $t\in\D$, say $t=k\,2^{-m}$, let us consider the set
\begin{equation*}\{q_t^{(n)}((x,y),.): \, n\ge m\}\subset \Pr(X^2).\end{equation*}
Being a subset of the set of all couplings of $p_t(x,.)$ and $p_t(y,.)$, this set is relatively compact, \cite{Vi2}, Lemma 4.4.
Thus there exists a coupling $q_t((x,y),.)$ and a subsequence $(n_l)_{l\in\N}$ such that
\begin{equation*}q_t=\lim_{l\to\infty} q_t^{n_l}\end{equation*}
(in the sense of weak convergence).
\end{remark}

\subsection{Coupled Stochastic Processes}

For the sequel, let us fix an initial distribution $\alpha\in \Pr(X^2)$ with marginals $\alpha_1=(e_1)_*\alpha$ and $\alpha_2=(e_2)_*\alpha$.
To simplify notation, in the following lemma and its proof we drop $\alpha$ from the notation.
For any $n\in\N_0$ and any finite subset $J$ of $\D_n$, say $J=\{t_1,\ldots,t_r\}$ with $t_1<\ldots <t_r$, we define a probability measure $Q_J^{(n)}$ on $(X^2)^{|J|}$ by
\begin{eqnarray*}
\lefteqn{Q_J^{(n)}\Large( A_1\times\ldots \times A_r\Large)=}\\
&=&
\int_{X^2} \int_{A_1}\ldots \int_{A_r} q^{(n)}_{t_r-t_{r-1}}((x_{r-1},y_{r-1}),d(x_r,y_r))\, \ldots \\
&&\qquad\qquad\qquad\qquad\qquad\qquad\ldots\, q^{(n)}_{t_1}((x_{0},y_{0}),d(x_1,y_1))\,
\alpha(d(x_0,y_0)).
\end{eqnarray*}
For fixed $n\in\N_0$, obviously $\{Q_J^{(n)}: \ J \mbox{ finite }\subset\D_n\}$ is a \emph{consistent} family of probability measures.

\begin{lemma}
For fixed finite $J\subset\D$, say $J\subset\D_m$, the family
$\{Q_J^{(n)}: \ n\in\N_0, \ n\ge m\}$
is a tight family of probability measures on $(X^2)^{|J|}$.
\end{lemma}

\begin{proof}
Let $J$ be given as $J=\{t_1,\ldots,t_r\}$ with $t_i\in\D_m$. Put
$P_{t_i}^{\alpha_1}(.)=\int_X p_{t_i}(x_0,.)\alpha_1(dx_0)$ and $P_{t_i}^{\alpha_2}(.)=\int_X p_{t_i}(y_0,.)\alpha_2(dy_0)$.
The families $\{P_{t_i}^{\alpha_1}: \ i=1,\ldots,r\}$ and $\{P_{t_i}^{\alpha_2}: \ i=1,\ldots,r\}$ of probability measures on $X$ are tight.
Thus, given $\epsilon>0$ there exist compact sets $A_1,A_2\subset X$ such that
\begin{equation*}P_{t_i}^{\alpha_1}(X\setminus A_1)<\epsilon,\quad P_{t_i}^{\alpha_2}(X\setminus A_2)<\epsilon \qquad(\forall i).\end{equation*}
Put  $\overrightarrow{A}=(A_1\times A_2)^r$. Then for all $n\in\N_0$
\begin{eqnarray*}
Q_J^{(n)}((X^2)^r\setminus \overrightarrow{A})&\le&
\sum_{i=1}^r Q_{t_i}^{(n)}(X^2\setminus A_1\times A_2)\\
&\le& \sum_{i=1}^r \left[Q_{t_i}^{(n)}((X\setminus A_1)\times X)+Q_{t_i}^{(n)}(X \times(X\setminus  A_2))\right] \\
&=& \sum_{i=1}^r \left[P_{t_i}^{\alpha_1}(X\setminus A_1)+P_{t_i}^{\alpha_2}(X\setminus  A_2)\right] \\
&\le& 2r\cdot \epsilon.
\end{eqnarray*}
Since $\overrightarrow{A}$ is a compact subset in $(X^2)^{|J|}$ this proves the claim.
\end{proof}

\begin{remark}\label{margi} Given $J=\{t_1,\ldots,t_r\}$ as above, define $\overrightarrow{e_1}: (X^2)^r\to X^r$ by
$((x_1,y_1),\ldots, (x_r,y_r))\mapsto (x_1,\ldots,x_r)$.
Then \begin{equation*}\left(\overrightarrow{e_1}\right)_*Q_J^{(n)}=P^{\alpha_1}_J\qquad(\forall n)\end{equation*}
where $P^{\alpha_1}_J$ is defined -- as before -- as a probability measure on $X^r$ by
\begin{equation*}P^{\alpha_1}_J(B_1\times\ldots B_r)= \int_{X} \int_{B_1}\ldots \int_{B_r} p_{t_r-t_{r-1}}(x_{r-1},dx_r)\, \ldots \, p_{t_1}(x_{0},dx_1)\,
\alpha_1(dx_0).\end{equation*}
Similarly, $(\overrightarrow{e_2})_*Q_J^{(n)}=P^{\alpha_2}_J$.
\end{remark}

\begin{proposition} (i) There exist a projective family
$\{Q_J^\alpha: \ J \mbox{ finite }\subset\D\}$ of probability measures
and  a subsequence $(n_l)_{l\in\N}$ such that for each finite $J\subset\D$
\begin{equation*}Q_J^{(n_l)}\to Q_J^\alpha\quad\mbox{weakly on }(X^2)^{|J|}\end{equation*}
as $l\to\infty$.

(ii) For each finite $J\subset\D$
\begin{equation*}\left(\overrightarrow{e_1}\right)_*Q_J^\alpha=P^{\alpha_1}_J\qquad\left(\overrightarrow{e_2}\right)_*Q_J^\alpha=P^{\alpha_2}_J\end{equation*}
where $P_J^{\alpha_1}$ and $P_J^{\alpha_2}$ denote the finite dimensional distribution of the heat flow on $X$ with initial distribution $\alpha_1$ and $\alpha_2$, resp.
\end{proposition}

\begin{proof}
(i) For fixed  $J$ the existence of a converging subsequence $Q_J^{(n_l)}, l\in\N$, follows from the tightness result of the previous lemma.
A diagonal sequence argument allows to choose this subsequence jointly for all finite $J\subset\D$.
The consistency condition is preserved under convergence.

(ii) is an immediate consequence of Remark \ref{margi}.
\end{proof}

The previous proposition together with Kolmogorov's extension theorem yields

\begin{corollary}
There exists a probability measure $\QQ^\alpha_\D$ on $(X^2)^\D$ such that for all finite $J\subset\D$
\begin{equation*}(\pi_J)_* \QQ^\alpha_\D=Q_J^\alpha.\end{equation*}
Moreover,
\begin{equation*}\left(\overrightarrow{e_1}\right)_*\QQ^\alpha_\D=\PP^{\alpha_1}_\D\qquad\left(\overrightarrow{e_2}\right)_*\QQ^\alpha_\D=\PP^{\alpha_2}_\D.\end{equation*}
\end{corollary}

Now let $\pi_t=(\pi^1_t,\pi^2_t): \ (X^2)^\D\to X^2,\ \omega\mapsto (\omega_1(t),\omega_2(t))$ be the coordinate process. Then under the measure $\QQ^\alpha_\D$
\begin{itemize}
\item
$(\pi^1_{t/2})_{t\in\D}$ is a Brownian motion on $X$ (restricted to dyadic time instances) with initial distribution $\alpha_1$;
\item
$(\pi^2_{t/2})_{t\in\D}$ is a Brownian motion on $X$ (restricted to dyadic time instances) with initial distribution $\alpha_2$.
\end{itemize}
It follows that under $\QQ^\alpha_\D$ the process $(\pi_t^1)_{t\in\D}$ has the local uniform continuity property (63.6) from \cite{Bauer}.
The same is true for the process $(\pi_t^2)_{t\in\D}$. Hence, also the joint process $\pi_t=(\pi^1_t,\pi_t^2), t\in\D,$ satisfies an analogous property.
Thus for $\QQ^\alpha_\D$-a.e. $\omega$
the limit
\begin{equation*}B_t=\lim_{s\to t, s\in\D} \pi_{s/2}\end{equation*}
exists for each $t\in\R_+$, it coincides with $\pi_{t/2}$ for each $t\in\D$ and  the trajectory $t\mapsto B_t(\omega)$ is continuous
(\cite{Bauer}, Lemma 63.5).

\begin{theorem}
Given any initial distribution $\alpha$ on $X^2$ then the stochastic process $(B_t)_{t\in\R_+}$ with values in $X^2$, constructed as above on the probability space $((X^2)^\D, \B(X^2)^\D, \QQ^\alpha_\D)$,
is a coupling of two Brownian motions $(B_t^1)_{t\in\R_+}$ and $(B_t^2)_{t\in\R_+}$ with values in $X$ and with initial distributions $\alpha_1$ and $\alpha_2$, resp., and it satisfies
\begin{equation}\label{pathwise}d(B^1_{s+t},B^2_{s+t})\le e^{-Kt/2} \cdot d(B^1_s,B^2_s) \qquad (\forall s,t\ge0)\end{equation}
for $\QQ^\alpha_\D$-a.e. path.
\end{theorem}

\begin{remark}
The construction of coupled Brownian motions satisfying the pathwise estimate (\ref{pathwise}) is well-known in the case of smooth Riemannian manifold, cf.
\cite{Wa} and references therein. In this case, it is most naturally constructed   using the Brownian motion on the frame bundle and estimates for the
stochastic parallel transport.
\end{remark}

\section{Variable Curvature Bounds}

For the remaining discussions, let us assume that $(X,d,m)$ is a mms which satisfies all the previous requirements (complete separable length space, full support, integrability condition (\ref{gaussian}), infinitesimally Hilbertian, condition C) and which satisfies one of the equivalent properties of Theorem \ref{fund} for some (large negative) $\KK\in\R$. In addition to this uniform bound $\KK$ we want to impose another variable curvature bound $\k$.
Here and in the sequel $\k$ will denote a lower semicontinuous function on $X$, bounded from below by the number $\KK$ and locally $m$-integrable.

\subsection{Bochner Inequality and Gradient Estimates}

Given such a function $\k$ we define the \emph{Schr\"odinger semigroup} $(T_t^{2\k})_{t\ge0}$ as the strongly continuous semigroup of operators on $L^2(X,m)$ with generator $\Delta -2\k$. (If $\k$ is locally bounded, this operator can be understood in any of the possible ways.
If $\k$ is merely $L^1_{loc}$ and bounded from below, the operator should be regarded as the \emph{form sum} of the Laplacian $\Delta$ and the multiplication operator $-2\k$.) See e.g. \cite{StoVoi}.
An explicit representation of this semigroup is given in terms of Brownian motion on $X$ by means of the Feynman-Kac formula:
\begin{equation*}T_t^{2\k}u(x)=\EE_x\left[e^{-\int_0^{2t} \k(B_s)ds}\cdot u(B_{2t})\right]\end{equation*}
for $m$-a.e. $x\in X$, each $t>0$ and each bounded Borel measurable $L^2$-function $u$ on $X$.
To see the strong continuity in $L^p$ note that $\int_0^t\k(B_s)ds\to0$ as $t\to0$ $\PP^x$-a.s. for $m$-a.e. $x\in X$ since $\k$ is locally integrable.
Thus  $M_t=\exp(-\int_0^{2t}\k(B_s)ds)$ is uniformly bounded from above by $e^{2|\KK|t}$ and a.s. continuous at $t=0$. Lebesgue's dominated convergence theorem therefore implies for each Borel $f\in L^p(X,m)$
\begin{eqnarray*}
\int_X \Big|T^{2\k}_tf- T_tf\Big|^p\,dm&\le&
\int_X \EE_x\left[\big| M_t-1\big|\cdot f(B_{2t})\right]^p\,dm(x)\\
&\le&
\int_X  \left(\int_X \EE_x\left[\big| M_t-1\big|^{p'}\right]^{p/p'}\,p_t(y,dx)\right)\cdot \big| f\big|^p(y)\,dm(y)\\
&\to&0
\end{eqnarray*}
as $t\to0$. Since in addition we know $\int_X \Big|T_tf-f\Big|^p\,dm\to0$ this proves the strong continuity in $L^p$.

\begin{theorem} For any mms $(X,d,m)$ as above, the following properties are equivalent
\begin{itemize}

\item[\bf(iv')] $L^2$-gradient estimate: $\forall u\in\Dom(\E)$ and $\forall t>0$
\begin{equation}\label{k-grad}
\Gamma(T_tu)\le T_t^{2\k}(\Gamma(u)).
\end{equation}
\item[\bf(v')]  Bochner's inequality or Bakry-Emery condition $\BE(\k,\infty)$:
 $\forall u\in\Dom(\Delta)$ with $\Delta u\in\Dom(\E)$
and $\forall \phi\in\Dom({\Delta})\cap L^\infty(X,m)$ with $\phi\ge0$, ${\Delta}\phi\in L^\infty(X,m)$
\begin{equation}\label{bochner2}
\int \left(\frac12\Delta -\k\right)\, \phi\cdot \Gamma(u)\,dm \ge\int \phi \cdot \Gamma(u,\Delta u)\,dm.
\end{equation}
\end{itemize}
Recall that according to Cor. 1  the square field operator $\Gamma(u)$ of a function $u\in\Dom(\E)$ coincides with the squared minimal weak upper gradient $|Du|_w^2$.
\end{theorem}

\begin{proof}
The proof follows the argumentation for constant curvature bounds.
Put \begin{equation*}f(s)=\int T_s^{2\k}\phi \cdot \Gamma(T_{t-s}u)\,dm=\int \phi \cdot T_s^{2\k} \Gamma(T_{t-s}u)\,dm\end{equation*}
and assume (v').
Then
\begin{eqnarray*}
f'(s)&=& \int (\Delta-2\k)T_s^{2\k}\phi\cdot \Gamma(T_{t-s}u)\,dm-
2\int T_s^{2\k}\phi\cdot \Gamma(T_{t-s}u, \Delta T_{t-s}u))\,dm\\
&\ge&0.
\end{eqnarray*}
(The regularity issues which guarantee that these calculations are applicable have been  discussed in detail in \cite{AGS-BE} and \cite{EKS}.)
Thus $f(t)\ge f(0)$. This is the claim in (iv'). Conversely, assume (iv'). Then $f(s)\ge f(0)$ for all $s>0$. Thus $f'(0)\ge0$.
That is,
\begin{equation*} \int (\Delta-2\k)\phi\cdot \Gamma(T_{t}u)\,dm-
2\int \phi\cdot \Gamma(T_{t}u, \Delta T_{t}u))\,dm\ge0\end{equation*}
for all $t>0$. By density (or passing to $t\to0$) this yields (v').
\end{proof}

\subsection{Curvature-Dimension~Condition~and~Evolution-Variation~Inequality}
Let us introduce the adaption of the $\CD(\KK,\infty)$-condition to non-constant curvature bounds.

\begin{definition}
We say that $(X,d,m)$ satisfies the \emph{curvature-dimension condition} $\CD(\k,\infty)$ with the curvature bound $\k:X\to\R$ if for every $\mu_0,\mu_1\in\Pr_2(X)\cap\Dom(\Ent)$ there exists
a geodesic $(\mu_t)_{t\in[0,1]}$ in $\Pr_2(X)$ connecting them and a probability measure $\Theta\in \Gamma(X)$ such that
$\mu_t=(e_t)_*\Theta$
and
\begin{equation}\label{k-conv}
\Ent(\mu_t)\le(1-t)\,\Ent(\mu_0)+t\,\Ent(\mu_1)-\int_0^1\int_{\Gamma(X)} g(s,t)\cdot \k(\gamma_s)\cdot |\dot\gamma|^2\,d\Theta(\gamma)\,ds
\end{equation}
for all $t\in[0,1]$. Here $e_t:\gamma\mapsto\gamma(t)$ denotes the evaluation map, $\Gamma(X)$ denotes the space of (constant speed, minimizing) geodesics in $X$ parametrized by $[0,1]$ and $|\dot\gamma|$ denotes the speed of such a geodesic. Moreover, $g(s,t)=\min\{s(1-t),t(1-s)\}$ is the Green function of the unit interval.
\end{definition}

Note that for constant $\k=\KK$
\begin{equation*}\int_0^1\int_{\Gamma(X)} g(s,t)\cdot \k(\gamma_s)\cdot |\dot\gamma|^2\,d\Theta(\gamma)\,ds=
K\cdot \frac{t(1-t)}2\cdot W^2_2(\mu_0,\mu_1).\end{equation*}

\begin{definition}
We say that $(X,d,m)$ satisfies the \emph{evolution-variation inequality} $\EVI_\k$ with the curvature bound $\k:X\to\R$ if for every  $\mu_0\in\Pr_2(X)$ there exists a curve $(\mu_t)_{t>0}$ in $\Dom(\Ent)$ with $\lim_{t\to0}\mu_t=\mu_0$ and for each $t>0$ and each $\nu\in\Pr_2(X)$ there exits a probability measure $\Theta_t$ on $\Gamma(X)$ such that $(e_s)_*\Theta_t$ for $s\in[0,1]$ defines a geodesic in $\Pr_2(X)$ connecting $\mu_t$ and $\nu$ and such that
    \begin{equation}\label{evi-k}\frac{d^+}{dt}\frac12 W_2^2(\mu_t,\nu)+
    \int_0^1 \int_{\Gamma(X)}(1-s)\,\k(\gamma_s)\,|\dot\gamma|^2\,d\Theta_t(\gamma)\,ds
    \le\Ent(\nu)-\Ent(\mu_t).
    \end{equation}
\end{definition}

Note that for both (\ref{k-conv}) and (\ref{evi-k}) the measures $\Theta$ and $\Theta_t$ are indeed unique \cite{RaSt}, cf. Cor. 1, provided one restricts the discussion to absolutely continuous measures $\mu_0,\mu_1$ or $\mu_t,\nu$, resp.

\begin{theorem}\label{cd-evi}
 For any mms $(X,d,m)$ as above, the following properties are equivalent
\begin{itemize}

\item[\bf(i')]
The curvature-dimension condition $\CD(\k,\infty)$;

\item[\bf(ii')]
The evolution-variation inequality $\EVI_\k$.

\end{itemize}
\end{theorem}

\begin{proof}
``(i')$\Longrightarrow$(ii')'':
Let us first note that for each geodesic $(\eta_r)_{r\in[0,1]}$ with the property (\ref{k-conv}) a straightforward calculation yields
\begin{equation}\label{ent-slope}
\frac{d^+}{dr}\Big|_{r=0}\Ent(\eta_r)\le \Ent(\eta_1)-\Ent(\eta_0)-\int_0^1\int_{\Gamma(X)} (1-s)\cdot \k(\gamma_s)\cdot |\dot\gamma|^2\,d\Theta(\gamma)\,ds.
\end{equation}
We apply this to the geodesic connecting $\mu_t=\eta_0$ and $\nu=\eta_1$.
Thanks to  \cite{AGMR}, Thm 6.3 and Prop. 6.6 we already know
\begin{equation*}\frac{d^+}{dt}\frac12 W_2^2(\mu_t,\nu)\le \frac{d^+}{dr}\Big|_{r=0}\Ent(\eta_r).\end{equation*}
(Originally, this was proven only for 'good' geodesics. However, according to \cite{RaSt} geodesics in $\Pr_2(X)$ are unique and thus 'good' - provided their endpoints are absolutely continuous measures.)
Combining this with the previous estimate yields the claim.

``(ii')$\Longrightarrow$(i')'':
Let an arbitrary geodesic $(\mu_r)_{r\in[0,1]}$ be given, fix $r\in(0,1)$,
and let  $(\mu_r^t)_{t>0}$ be the solution to the $\EVI_\k$-flow starting at $\mu_r$.
Obviously, any $\EVI_\k$-flow is also an $\EVI_\KK$-flow for any constant $\KK\le\k$ and thus coincides with the heat flow $(P_t\mu_r)_{t>0}$ starting in $\mu_r$. Now apply the $\EVI_\k$-inequality (\ref{evi-k})
to $\mu_r^t$ and $\nu=\mu_0$ (or $\nu=\mu_1$, resp.).
Then
  \begin{equation}\label{eins}\frac{d^+}{dt}\frac12 W_2^2(\mu_r^t,\mu_0)+
    \int_0^1 \int_{\Gamma(X)}(1-s)\,\k(\gamma_s)\,|\dot\gamma|^2\,d\Theta_t^0(\gamma)\,ds
    \le\Ent(\mu_0)-\Ent(\mu_r^t)
    \end{equation}
and
  \begin{equation}\label{zwei}\frac{d^+}{dt}\frac12 W_2^2(\mu_r^t,\mu_1)+
    \int_0^1 \int_{\Gamma(X)}(1-s)\,\k(\gamma_s)\,|\dot\gamma|^2\,d\Theta_t^1(\gamma)\,ds
    \le\Ent(\mu_1)-\Ent(\mu_r^t)
    \end{equation}
    where $\Theta_t^0$ and $\Theta_t^1$ denote probability measures on $\Gamma(X)$ such that
    $(e_s)_*\Theta_t^0, s\in[0,1]$  and $(e_s)_*\Theta_t^1, s\in[0,1]$ are geodesics connecting $\mu_r^t$ with  $\mu_0$ or $\mu_1$, resp.
  Since $(\mu_r)_{r\in[0,1]}$ is a geodesic
\begin{eqnarray*}(1-r)\,W_2^2(\mu_0,\mu_r^t) +r\,W_2^2(\mu_r^t,\mu_1)&\ge& (1-r)r\,W_2^2(\mu_0,\mu_1)\\
&=&(1-r)\,W_2^2(\mu_0,\mu_r)+r\,W_2^2(\mu_r,\mu_1)
\end{eqnarray*}
for all $t\ge0$ and thus
\begin{equation*}\label{both}(1-r)\cdot\frac{d^+}{dt}\frac12 W_2^2(\mu_r^t,\mu_0)+r\cdot\frac{d^+}{dt}\frac12 W_2^2(\mu_r^t,\mu_1)\ge0.\end{equation*}
Adding up (\ref{eins}) -- multiplied by $(1-r)$ --  and (\ref{zwei}) -- multiplied by $r$ -- and taking into account (\ref{both}) therefore yields
 \begin{eqnarray*}
 \lefteqn{(1-r) \Ent(\mu_0)+r\Ent(\mu_1)-\Ent(\mu_r^t)}\\
 & \ge&
    (1-r)\int_0^1 \int_{\Gamma(X)}(1-s)\,\k(\gamma_s)\,|\dot\gamma|^2\,d\Theta_t^0(\gamma)\,ds\\
  &&  +r \int_0^1 \int_{\Gamma(X)}(1-s)\,\k(\gamma_s)\,|\dot\gamma|^2\,d\Theta_t^1(\gamma)\,ds.
      \end{eqnarray*}
Now let us consider the limit $t\to0$. Lower semicontinuity of $\Ent$ allows to pass to the limit on the LHS of the previous estimate.
Passing to the limit on the RHS is justified by the fact that
\begin{equation*}\Theta\mapsto \int_0^1\int_{\Gamma(X)}(1-s)\,\k(\gamma_s)\,|\dot\gamma|^2\,d\Theta(\gamma)\,ds\end{equation*}
is a lower semicontinuous function on the space of probabilty measure on $\Gamma(X)$.
Thus
 \begin{eqnarray}\label{drei}
\lefteqn{ (1-r) \Ent(\mu_0)+r\Ent(\mu_1)-\Ent(\mu_r)}\nonumber\\
 & \ge&
    (1-r)\int_0^1 \int_{\Gamma(X)}(1-s)\,\k(\gamma_s)\,|\dot\gamma|^2\,d\Theta^0(\gamma)\,ds\nonumber\\
  &&  +r \int_0^1 \int_{\Gamma(X)}(1-s)\,\k(\gamma_s)\,|\dot\gamma|^2\,d\Theta^1(\gamma)\,ds
      \end{eqnarray}
  where $\Theta^0$ and $\Theta^1$ now denote probability measures on $\Gamma(X)$ which represent geodesics in $\Pr_2(X)$ connecting
  $\mu_r$ with $\mu_0$ and $\mu_1$, resp.

  Now we will reparametrize and glue together the measures $\Theta^0$ and $\Theta^1$. First we define the measure $\hat\Theta^0:=(\Psi^0)_*\Theta^0$ on the space of geodesics $\gamma:[0,r]\to X$ and the measure $\hat\Theta^1=(\Psi^1)_*\Theta^1$ on the space of geodesics $\gamma:[r,1]\to X$ by means of the rescaling maps
   \begin{equation*}
   (\Psi^0\gamma)_s=\gamma_{1- s/r},\qquad (\Psi^1\gamma)_s=\gamma_{(s-r)/(1-r)}.
  \end{equation*}
  Note that the evaluation map at time $r$ for both measures leads to the same projection
   \begin{equation*} (e_r)_* \Theta^0=\mu_r=(e_r)_* \Theta^1.
   \end{equation*}
  Disintegration w.r.t. $\mu_r$ yields Markov kernels $\vartheta^0(y,d\gamma)$ and $\vartheta^1(y,d\gamma)$
 such that
  \begin{equation*}
 \hat\Theta^i(d\gamma)=\int_X \vartheta^i(y,d\gamma)\,\mu_r(dy)
 \end{equation*}
 for $i=0,1$.
 Let $\Gamma_r(X)$ denote the space of constant speed curves $\gamma: [0,1]\to X$  for which $\gamma^0:=\gamma\big|_{[0,r]}$ and $\gamma^1:=\gamma\big|_{[r,1]}$ are geodesics.
 Define a probability measure $\Theta$ on $\Gamma_r(X)$ by
   \begin{equation*}
 \Theta(d\gamma)=\int_X \vartheta^0(y,d\gamma^0)\,\vartheta^1(y,d\gamma^0)\,\mu_r(dy).
 \end{equation*}
 By construction, this measure $\Theta$ lives on the set of piecewise geodesic curves. Our claim, however, is that it is indeed supported by the set of geodesics.
 Let us consider the curve $\tilde\mu_s=(e_s)_*\Theta$, $s\in[0,1]$, which connects $\mu_0$ and $\mu_1$. Moreover,
 $W_2(\tilde \mu_0,\tilde \mu_r)=W_2(\mu_0, \mu_r)$ and $W_2(\tilde\mu_r,\tilde\mu_1)=W_2(\mu_r,\mu_1)$. Thus $(\tilde\mu_s)_{s\in[0,1]}$ is a geodesic in $\Pr_2(X)$ and therefore
$\Theta$ is supported by the set of geodesics in $X$.

 In terms of this probability measure $\Theta$ on $\Gamma(X)$, (\ref{drei}) can be rewritten as
 \begin{eqnarray*}
 (1-r) \Ent(\mu_0)+r\Ent(\mu_1)-\Ent(\mu_r)
 & \ge&
\int_0^1 \int_{\Gamma(X)}g(r,s)\,\k(\gamma_s)\,|\dot\gamma|^2\,d\Theta(\gamma)\,ds.
      \end{eqnarray*}
 This is the claim for a given $r\in[0,1]$. A priori the measure $\Theta$ might depend on $r$. Uniqueness of the optimal plan $\Theta$, however, implies then that the claim holds for every $r$.
\end{proof}

\subsection{From Bakry-Emery  to Curvature-Dimension}

The following 
action estimate in the spirit of \cite{AGS-BE}, Theorem 4.16, is regarded as a key ingredient for proving that the Bakry-Emery condition $\BE(\k,\infty)$ implies the curvature-dimension condition $\CD(\k,\infty)$ and the evolution-variation inequality $\EVI_\k$.

\begin{lemma} {\bf i)} Let $\rho_s=f_s\,m$, $s\in[0,1]$ be a regular curve (in the sense of \cite{AGS-BE}, Def. 4.10) and put $\rho_{s,t}=P_{st}\rho_s$.
Then for every Lipschitz function $\phi$ with bounded support and for all $t>0$
\begin{eqnarray*}
\int\phi_1\,d\rho_{1,t}-\int\phi_0\,d\rho_0&-&\int_0^1|\dot\rho_s|^2\,ds+2t\Big(\Ent(\rho_{1,t})-\Ent(\rho_0)\Big)\nonumber\\
&\le&
-\int_0^1\int_0^{st}\int_X T_r\left(2\k\cdot T_{st-r}^{2\k}\Gamma(\phi_s)\right)\,d\rho_s\,dr\,ds
\end{eqnarray*}
where $\phi_s=Q_s\phi$, $s\in[0,1]$, denote the Hamilton-Jacobi flow induced by $\phi$.

{\bf ii)} The same holds true for every geodesic $(\rho_s)_{s\in[0,1]}$ in $\Pr_2(X)$ with absolutely continuous measures in which case  $\int_0^1|\dot\rho_s|^2\,ds=W^2_2(\rho_0,\rho_1)$.
\end{lemma}

\begin{proof}
{\bf i)}
For $\k=0$, obviously the RHS vanishes and the estimate coincides with the result in \cite{AGS-BE}, Thm 4.16.
There is only one estimate which changes if one passes from $\k=0$ to $\k\not=0$: In the previous case, the gradient estimate allows to estimate
\begin{eqnarray*}
-\int_0^1\int_X  T_{st}\Gamma(\phi_s+t\,g^\epsilon_{s,t})\cdot f_s\,dm\,ds\end{eqnarray*}
in terms of
\begin{eqnarray}\label{old}
-\int_0^1\int_X  \Gamma(T_{st}(\phi_s+t\,g^\epsilon_{s,t}))\cdot f_s\,dm\,ds.\end{eqnarray}
(For the definition of $g^\epsilon_{s,t}$ we refer to \cite{AGS-BE}.) In the case $\k\not=0$, our new gradient estimate (\ref{k-grad}) will allow to estimate
\begin{eqnarray*}
-\int_0^1\int_X  T^{2\k}_{st}\Gamma(\phi_s+t\,g^\epsilon_{s,t})\cdot f_s\,dm\,ds\end{eqnarray*}
in terms of
the same expression (\ref{old}) as before.
Thus
\begin{eqnarray}
\int\phi_1\,d\rho_{1,t}-\int\phi_0\,d\rho_0&-&\int_0^1|\dot\rho_s|^2\,ds+2t(\Ent(\rho_{1,t})-\Ent(\rho_0))\nonumber \\
&\le&
-\int_0^1\int_X  T_{st}\Gamma(\phi_s+t\,g^\epsilon_{s,t})\cdot f_s\,dm\,ds\nonumber\\
&&+
\int_0^1\int_X  T^{2\k}_{st}\Gamma(\phi_s+t\,g^\epsilon_{s,t})\cdot f_s\,dm\,ds.
\label{rest}
\end{eqnarray}
Duhamel's principle states that
\begin{equation*}T_{st}-T^{2k}_{st}=\int_0^{st}T_r\big(2\k\cdot T_{st-r}^{2\k}\big)\,dr\end{equation*}
in the sense of operators.
Thus the RHS of (\ref{rest}) can be estimated by
\begin{eqnarray*}
-\int_0^1\int_0^{st}\int_X T_r\left(2\k\cdot T_{st-r}^{2\k}\Gamma(\phi_s+t\,g^\epsilon_{s,t})\right)\cdot f_s\,dm\,dr\,ds.
\end{eqnarray*}
Thanks to the uniform estimates and convergence properties of $g^\epsilon_{s,t}$ established in \cite{AGS-BE},  we obtain in the limit $\epsilon\to0$
\begin{eqnarray*}
\lefteqn{\int\phi_1\,d\rho_{1,t}-\int\phi_0\,d\rho_0-\int_0^1|\dot\rho_s|^2\,ds+2t(\Ent(\rho_{1,t})-\Ent(\rho_0))}\nonumber \\
&\le&
-\int_0^1\int_0^{st}\int_X T_r\left(2\k\cdot T_{st-r}^{2\k}\Gamma(\phi_s)\right)\cdot f_s\,dm\,dr\,ds
\end{eqnarray*}
which is the first claim.

\smallskip

{\bf ii)}
Given any  geodesic $(\rho_s)_{s\in[0,1]}$ with absolutely continuous measures $\rho_s=f_s\,m$, we approximate it by regular curves $(\rho_s^n)_{s\in[0,1]}$. The measures  $\rho_s^n$ will be absolutely continuous with densities  $f_s^n$. We may choose the approximation such that $f^n_s(x)\to f_s(x)$ for a.e. $(x,s)\in X\times[0,1]$ as $n\to\infty$.
Put $\rho_{1,t}^n=\mu_t^n=P_t\mu_0^n$ and let $\phi_s=Q_s\phi$ be the Hamilton-Jacobi flow induced by any Lipschitz function $\phi$ with bounded support. Part i) of the lemma then yields
\begin{eqnarray*}
\int\phi_1\,d\rho^n_{1,t}&-&\int\phi_0\,d\rho^n_0
-\int_0^1|\dot\rho_s^n|^2\,ds+2t\Big(\Ent(\rho_{1,t}^n)-\Ent(\rho_0^n)\Big)\nonumber\\
&\le&
-\int_0^1s\int_0^{t}\int_X T_{sr}\left(2\k\cdot T_{s(t-r)}^{2\k}\Gamma(\phi_s)\right)\cdot f_s^n\,dm\,dr\,ds.
\end{eqnarray*}
Of course, $\int\phi_1\,d\rho^n_{1,t}\to \int\phi_1\,d\rho_{1,t}$ and $\int\phi_0\,d\rho^n_0\to \int\phi_0\,d\rho_0$ as $n\to\infty$
as well as $\int_0^1|\dot\rho_s^n|^2\,ds\to W^2_2(\rho_0,\rho_1)$.
Lower semicontinuity of the entropy implies $\Ent(\rho_{1,t})\le\liminf_{n\to\infty} \Ent(\rho_{1,t}^n)$ and by construction of the regular curve we may achieve that $\Ent(\rho_0)\ge\limsup_{n\to\infty} \Ent(\rho_0^n)$.
Passing to the limit $n\to\infty$ the previous estimate thus yields
\begin{eqnarray}
\int\phi_1\,d\rho_{1,t}&-&\int\phi_0\,d\rho_0-W_2^2(\rho_0,\rho_{1})+2t\Big(\Ent(\rho_{1,t})-\Ent(\rho_0)\Big)\nonumber\\
&\le&
-\limsup_{n\to\infty}\int_0^1s\int_0^{t}\int_X T_{sr}\left(2\k\cdot T_{s(t-r)}^{2\k}\Gamma(\phi_s)\right)\cdot f_s^n\,dm\,dr\,ds\nonumber\\
&\le&
-\int_0^1s\int_0^{t}\int_X T_{sr}\left(2\k\cdot T_{s(t-r)}^{2\k}\Gamma(\phi_s)\right)\cdot f_s\,dm\,dr\,ds.\label{april}
\end{eqnarray}
where the last inequality follows by means of Fatou's lemma from lower boundedness of \\ $T_{sr}\left(2\k\cdot T_{s(t-r)}^{2\k}\Gamma(\phi_s)\right)$  and from the a.e. convergence of $f_s^n\to f_s$ (together with $\int f_s^n\,dm=\int f_s\, dm$).
\end{proof}

\begin{lemma}
Let $\rho_0$ and $\rho_1\in\Pr_2(X)\cap\Dom(\Ent)$ be given (without restriction, we may assume that both measures have bounded densities), let $(\rho_s)_{s\in[0,1]}$ be the unique geodesic connecting them and put
$\rho_{1,t}=P_t\rho_1$.

Fix  a Kantorovich potential $\phi_0$ for the optimal transport from $\rho_0$ to $\rho_{1}$ and for each $t>0$ let $\phi_0^t$ be a
Lipschitz function with bounded support such that $\int\phi_1^t\,d\rho_{1,t}-\int\phi_0^t\,d \rho_0\ge W^2_2(\rho_0,\rho_{1,t})-2t^2$
(i.e. $\phi_0^t$ is 'almost optimal' for the Kantorovich duality problem for $\rho_0$ and $\rho_{1,t}$) and
\[ \Gamma(\phi_s^t)(x)\to \Gamma(\phi_s)(x) \qquad  \mbox{in }L^1(X\times [0,1],dm\otimes ds)\]
as $t\to0$ where $\phi_s^t=Q_s\phi^t_0$ denotes the Hamilton-Jacobi flow induced by $\phi^t_0$. 
Then
\begin{eqnarray*}
\lefteqn{\limsup_{t\to0}\frac1{2t}\Big[W_2^2(\rho_0,\rho_{1,t})-W_2^2(\rho_0,\rho_{1})\Big] } \\
&
\le&\Ent(\rho_{0})-\Ent(\rho_1)-\int_0^1s\int_X \Gamma(\phi_s)\cdot \k\cdot  f_s\,dm\,ds.
\end{eqnarray*}
\end{lemma}

\begin{proof}
 Due to the $\CD(\KK,\infty)$-condition we know that all the measures $\rho_s$ will be absolutely continuous with uniformly bounded densities, say $f_s\le C_0$.

Applying the estimate of the previous Lemma to the function $\phi^t_0$ (inducing the transport almost towards the measure $\rho_{1,t}$) and dividing by $-2t$ gives
\begin{eqnarray*}
-\frac1{2t}\Big[W_2^2(\rho_0,\rho_{1,t})&-&W_2^2(\rho_0,\rho_{1})\Big]+t-\Big(\Ent(\rho_{1,t})-\Ent(\rho_0)\Big)\nonumber\\
&\ge&
\frac1{2t}\int_0^1s\int_0^{t}\int_X \Gamma(\phi_s^t)\cdot T_{s(t-r)}^{\k}\left(2\k\cdot T_{sr} f_s\right)\,dm\,dr\,ds\\
&\ge&
\int_0^1s\int_X\Gamma(\phi_s^t)\cdot\Big[ \frac1t\int_0^{t}T_{s(t-r)}^{2\k}\left(\big(\k\wedge C_1\big)\cdot T_{sr} f_s\right)\,dr \Big] \,dm\,ds
\end{eqnarray*}
for any $C_1\ge0$.
Strong continuity of the semigroups $(T_r)_{r>0}$ and $(T^{2\k}_r)_{r>0}$ in $L^p$ implies that for any $p\in [1,\infty)$
\[ \frac1t\int_0^{t}T_{s(t-r)}^{2\k}\left(\big(\k\wedge C_1\big)\cdot T_{sr} f_s\right)
\to \big( \k\wedge C_1\big)\cdot f_s\qquad \mbox{in }L^p\]
as $t\to0$
and (after passing to a subsequence -- which we drop from the notation) also a.e. on $X\times[0,1]$. Moreover, note that
\[ \left|\frac1t\int_0^{t}T_{s(t-r)}^{2\k}\left(\big(\k\wedge C_1\big)\cdot T_{sr} f_s\right)\right|\le
\big(|K|\vee C_1\big)\cdot C_0\cdot e^{2|K|}\]
for a.e. $(x,s)$ and all $t\le1$.
Together with the $L^1$-convergence $\Gamma(\phi_s^t)\to\Gamma(\phi_s)$
this gives that
\[ \Gamma(\phi_s^t)\cdot \frac1t\int_0^{t}T_{s(t-r)}^{2\k}\left(\big(\k\wedge C_1\big)\cdot T_{sr} f_s\right)
\to  \Gamma(\phi_s)\cdot\big( \k\wedge C_1\big)\cdot f_s\]
in $L^1(X\times [0,],dm\otimes ds)$ as $t\to0$. Thus
\begin{eqnarray}
-\limsup_{t\to0}\frac1{2t}\Big[W_2^2(\rho_0,\rho_{1,t})&-&W_2^2(\rho_0,\rho_{1})\Big]-\Big(\Ent(\rho_{1})-\Ent(\rho_0)\Big)\nonumber\\
&\ge&
\int_0^1s\int_X \Gamma(\phi_s)\cdot \big(\k\wedge C_1\big)\cdot  f_s\,dm\,ds\label{may}
\end{eqnarray}
for all $C_1\ge0$. Monotone convergence allows to pass to the limit $C_1\to\infty$ in (\ref{may}). Thus
\begin{eqnarray*}
-\limsup_{t\to0}\frac1{2t}\Big[W_2^2(\rho_0,\rho_{1,t})&-&W_2^2(\rho_0,\rho_{1})\Big]-\Big(\Ent(\rho_{1})-\Ent(\rho_0)\Big)\nonumber\\
&\ge&
\int_0^1s\int_X \Gamma(\phi_s)\cdot \k\cdot  f_s\,dm\,ds
\end{eqnarray*}
which is the claim.
\end{proof}

\subsection{Local-to-Global and Stability}
Recall our assumption that $(X,d,m)$ is infinitesimally Hilbertian.
\begin{theorem}  If the  curvature-dimension condition $\CD(\k,\infty)$ holds locally then it also holds globally.
\end{theorem}

Here we say that the  $\CD(\k,\infty)$-condition holds locally on $X$ if $X$ can be covered by open sets $U_i$  such that for each pair of measures
$\mu_0,\mu_1$
supported in $\overline U_i$ a connecting geodesic $(\mu_t)_{t\in[0,1]}$ exists (which might leave the set $\overline U_i$)
with property (\ref{k-conv}).

\begin{proof}
According to \cite{RaSt} we know that an optimal transport never charges branching geodesics.
This allows to pass to a ``pathwise version'' of (\ref{k-conv}), cf. \cite{Stu1} (``localization in space''): for $\Theta$-a.e. $\gamma\in\Gamma(X)$ and every $t\in[0,1]$
\begin{equation}\label{k-conv-path}
\log\rho_t(\gamma_t)\le(1-t)\,\log\rho_0(\gamma_0)+t\,\log\rho_1(\gamma_1)-\int_0^1 g(s,t)\cdot \k(\gamma_s)\cdot |\dot\gamma|^2\,ds.
\end{equation}
This in turn is nothing but an integrated version of the fact that
$t\mapsto \rho_t(\gamma_t)$ is upper semicontinuous on $[0,1]$, continuous on $(0,1)$ and satisfies the differential inequality $\frac{\partial ^2}{\partial t^2}\log\rho_t(\gamma_t)\ge k(\gamma_t) |\dot\gamma|^2$ weakly
on $(0,1)$. This obviously allows for a ``localization in time''.
The proof thus follows by a careful adaption of the arguments in \cite{Stu1}.
\end{proof}

As a consequence of the previous local-to-global result we obtain that the curvature-dimension condition $\CD(\k,\infty)$ is stable under tensorization.

\begin{corollary}
Assume that mms $(X_i,d_i,m_i)$ for $i=1,\ldots,n$ are given, each of  which satisfies a curvature-dimension condition $\CD(\k_i,\infty)$ for some lower semicontinuous, lower bounded function $\k_i:X_i\to\R$.
Then the product space
\begin{equation*}\big(X,d,m\big)=\Big( X_1\times\ldots\times X_n, \sqrt{d_1^2+\ldots+d_n^2}, m_1\otimes\ldots m_n\Big)\end{equation*}
satisfies the curvature-dimension condition $\CD(\k,\infty)$ for the function
\begin{equation*}\k(x_1,\ldots,x_n):=\min\{0,\k_1(x_1),\ldots,\k_n(x_n)\}.\end{equation*}
\end{corollary}

\begin{proof}
Let us first prove the claim in the particular case where all the functions $\k_i$ are continuous.
According to the previous theorem, it suffices to prove that the  $\CD(\k,\infty)$-condition holds locally on $X$.
Given $x\in X$ and $\epsilon>0$ there exists open neighborhoods $V_i\supset U_i$ of $x_i$ and constants $\KK_i$ such that
$V_i$ contains the convex hull of $\overline U_i$ (i.e.
geodesics with endpoints in $\overline U_i$ do not leave $V_i$) and such that
\begin{equation*}\k_i(y)\ge \KK_i\ge \k_i(z)-\epsilon\end{equation*}
for $y,z\in V_i$. Put
$\KK:=\min\{0,\KK_1,\ldots,\KK_n\}$. Then for optimal transports in $X_i$ with marginals $\mu_0^i,\mu_1^i$ supported in $\overline U_i$ the `classical' $\CD(\KK_i,\infty)$-condition (with constant curvature bound $\KK_i$) applies. Due to the tensorization property of the latter \cite{Stu1},
the $\CD(\KK,\infty)$-condition (with constant curvature bound $\KK$ defined as above) applies to optimal transports in $X$ with marginals $\mu_0,\mu_1$ supported in $\overline U=\overline U_1\times\ldots\times \overline U_n$.
Thus the $\CD(\k-\epsilon,\infty)$-condition holds locally on $X$.
According to the previous local-to-global theorem this implies that the  $\CD(\k-\epsilon,\infty)$-condition holds globally on $X$.
Passing to the limit $\epsilon\to0$ yields (via monotone convergence)  the  $\CD(\k,\infty)$-condition globally on $X$.
This proves the claim in the particular case of continuous $\k_i$.

Now let us treat the case of general $\k_i$. For each $i$ there exists an monotone sequence of continuous functions $\k_i^l:X_i\to\R$ such that
\begin{equation*}\k_i^l(x)\nearrow \k_i(x)\qquad(\forall x\in X_i)\end{equation*}
as $l\to\infty$.
Applying the previous result to the functions $(\k_i^l)_{i=1\ldots,n}$ for fixed $l$ yields that $X$ satisfies
the $\CD(\k^l,\infty)$-condition  globally on $X$ with $\k^l(x_1,\ldots,x_n):=\min\{0,\k^l_1(x_1),\ldots,\k^l_n(x_n)\}$.
In the limit $l\to\infty$, the $\CD(\k,\infty)$-condition follows by monotone convergence.
\end{proof}

\begin{remark} Obviously, also the $\BE(\k,\infty)$-condition is stable under tensorization.
\end{remark}

Our next goal is to analyze how the curvature-dimension condition $\CD(\k,\infty)$ behaves  under change of measure.

Given  functions $V$ and $\lambda$ on $X$, we say that $V$ is strongly $\lambda$-convex if for every geodesic $\gamma\in\Gamma(X)$ and every $t\in[0,1]$
\begin{equation}\label{V-conv}
V(\gamma_t)\le(1-t)\, V(\gamma_0)+t\, V(\gamma_1)-\int_0^1 g(s,t)\cdot \lambda(\gamma_s)\cdot |\dot\gamma|^2\,ds.
\end{equation}

\begin{proposition}
If $(X,d,m)$ satisfies the curvature-dimension condition $\CD(\k,\infty)$ and if $V:X\to\R$ is strongly $\lambda$-convex
then
$(X,d,e^{-V}\,m)$ satisfies the curvature-dimension condition $\CD(\k+\lambda,\infty)$.
\end{proposition}

\begin{proof}
Recall that the relative entropy $\Ent'$ w.r.t. $m'=e^{-V}m$ is given by
\begin{equation*}\Ent'(\mu)=\Ent(\mu)+\int V\,d\mu.\end{equation*}
Integrating (\ref{V-conv}) w.r.t. the optimal path measure $\Theta$ from (\ref{k-conv})
leads to
\begin{equation}
\int V\,\mu_t\le(1-t)\, \int V\,d\mu_0+t\, \int V\,d\mu_1-\int_0^1\int_{\Gamma(X)} g(s,t)\cdot \lambda(\gamma_s)\cdot |\dot\gamma|^2\,d\Theta(\gamma)\,ds.
\end{equation}
Adding this to (\ref{k-conv}) yields the claim.
\end{proof}

Finally, we want to study whether the $\CD(\k,\infty)$-condition is stable under convergence. The precise formulation of this question already  requires some care: the curvature bounds  for the approximating spaces and for the limit space will be functions defined on different spaces.
To avoid additional complications,
we will restrict the discussion in the sequel to \emph{normalized} mms, i.e. mms $(X,d,m)$ with $m(X)=1$.

Recall that a sequence of normalized mms $(X_n,d_n,m_n)$, $n\in\N$, converges to a mms $(X,d,m)$ in $L^2$-transportation distance $\DD$
if and only if there exists a metric space $(X^*,d^*)$ and isometric embeddings
$\iota_n:\, X_n\hookrightarrow X^*$,   $\iota:\, X\hookrightarrow X^*$ such that the push forward measures converge w.r.t. $L^2$-Wasserstein distance $W^*_2$ on $(X^*,d^*)$:
\begin{equation*}W_2^*\Big((\iota_n)_*m_n, \iota_*m\Big)\to0.\end{equation*}
\begin{definition}
We say that a function $\k:X\to\R$ is \emph{asymptotically dominated} by a sequence of functions $\k_n:X_n\to\R$ if for each $\epsilon>0$ there exists $n'\in\N$, isometric embeddings $(\iota_n)_{n\ge n'},\iota$ into a common metric space $(X^*,d^*)$ as above and a lower semicontinuous function $\k^*:\, X^*\to\R$ such that \, $\forall n\ge n'$
\begin{equation}
\k\le \k^*\circ\iota \quad\mbox{on }X,\qquad \k^*\circ\iota_n\le \k_n+\epsilon \quad\mbox{on }X_n .
\end{equation}
\end{definition}

\begin{theorem}
For each $n\in\N$, let  $(X_n,d_n,m_n)$ be a normalized mms which satisfies the curvature-dimension condition $\CD(\k_n,\infty)$
for some function $\k_n:X_n\to\R$. Assume that for $n\to\infty$ the sequence of spaces $(X_n,d_n,m_n)$ converges in $\DD$-distance to some normalized
mms $(X,d,m)$ and assume that
 the function $\k:X\to\R$ is asymptotically dominated by the sequence of functions $\k_n:X_n\to\R$.
 Then $(X,d,m)$ satisfies the curvature-dimension condition $\CD(\k,\infty)$.
\end{theorem}
In brief words: The curvature-dimension condition $\CD(\k,\infty)$ is stable under convergence.

\begin{proof}
Our proof follows the argumentation in \cite{Stu1}.
Given the mms $(X_n,d_n,m_n)$, $n\in\N$, and $(X,d,m)$ as above we may assume without restriction that they are already isometrically embedded as subsets into some space $(X^*,d^*)$. Let $q_n$ be an $W_2^*$-optimal coupling  of $m_n$ and $m$.
Its disintegration kernel allows to map each probability measure $\mu$ which is absolutely continuous w.r.t. $m$ onto a probability measure $\mu^{n}$ which is absolutely continuous w.r.t. $m_n$ in such a way that $\Ent_n(\mu^n)\le \Ent(\mu)$.

Given two probability measures $\mu_0,\mu_1$ (supported on $X$) with finite entropy w.r.t. $m$ we thus obtain in a canonical way corresponding
probability measures $\mu_0^n,\mu_1^n$ (supported on $X_n$) with finite entropy w.r.t. $m_n$.
The curvature-dimension condition $\CD(\k_n,\infty)$ for the space  $(X_n,d_n,m_n)$
yields the existence of a probability measure $\Theta^n$ on $\Gamma(X_n)\subset\Gamma(X^*)$ which induces  a geodesic $\mu_t^n=(e_t)_*\Theta^n$ connecting $\mu_0^n,\mu_1^n$ and which satisfies
 \begin{eqnarray*}
\Ent_n(\mu_t^n)&\le&(1-t)\,\Ent_n(\mu_0^n)+t\,\Ent_n(\mu_1^n)\\
&&-\int_0^1\int_{\Gamma(X_n)} g(s,t)\cdot \k_n(\gamma_s)\cdot |\dot\gamma|^2\,d\Theta^n(\gamma)\,ds
\end{eqnarray*}
for each $t$.
Tightness implies
the existence of a converging subsequence -- again denoted by $(\Theta^n)_{n\ge n'}$ -- and a limit measure $\Theta\in\Pr(\Gamma(X^*))$.
Lower semicontinuity of $n\mapsto \Ent_n(\mu_t^n)$ (w.r.t. both measures involved) then provides the estimate
\begin{equation*}\Ent(\mu_t)\le\liminf_{n\to\infty} \Ent_n(\mu_t^n).\end{equation*}
Moreover, by construction we have $\Ent(\mu_t)=\lim_{n\to\infty} \Ent_n(\mu_t^n)$ for $t=0$ and $t=1$.
It remains to prove
 \begin{eqnarray*}
\lefteqn{\liminf_{n\to\infty}\int_0^1\int_{\Gamma(X_n)} g(s,t)\cdot \k_n(\gamma_s)\cdot |\dot\gamma|^2\,d\Theta^n(\gamma)\,ds}\\
&\ge&
\int_0^1\int_{\Gamma(X)} g(s,t)\cdot \k(\gamma_s)\cdot |\dot\gamma|^2\,d\Theta(\gamma)\,ds.
\end{eqnarray*}
Using the embedding into $X^*$ and the estimates between $\k_n, \k^*$ and $\k$ it suffices to prove
 \begin{eqnarray*}
\lefteqn{\liminf_{n\to\infty}\int_0^1\int_{\Gamma(X^*)} g(s,t)\cdot \k^*(\gamma_s)\cdot |\dot\gamma|^2\,d\Theta^n(\gamma)\,ds}\\
&\ge&
\int_0^1\int_{\Gamma(X^*)} g(s,t)\cdot \k^*(\gamma_s)\cdot |\dot\gamma|^2\,d\Theta(\gamma)\,ds.
\end{eqnarray*}
This finally follows from the weak convergence $\Theta^n\to\Theta$ and from the fact that
 \begin{equation*}
\Theta\mapsto
\int_0^1\int_{\Gamma(X^*)} g(s,t)\cdot \k^*(\gamma_s)\cdot |\dot\gamma|^2\,d\Theta(\gamma)\,ds
\end{equation*}
is a lower semicontinuous function on $\Pr(\Gamma(X^*))$ which in turn follows from the lower semicontinuity of $\k^*$.
\end{proof}

\end{document}